\documentclass[11pt]{amsart}

\usepackage{amsfonts, amsmath, amscd}
\usepackage[psamsfonts]{amssymb}

\usepackage{amssymb}

\usepackage{pb-diagram}

\usepackage[all,cmtip]{xy}

\usepackage[usenames]{color}

\newtheorem{theorem}{Theorem}[section]
\newtheorem{lemma}[theorem]{Lemma}
\newtheorem{corollary}[theorem]{Corollary}

\newtheorem{remark}[theorem]{Remark}
\newtheorem{proposition}[theorem]{Proposition}

\numberwithin{equation}{section}

\newcommand{\CC}{C_k}
\newcommand{\NN}{\mathbb{N}}

\newcommand{\CP}{\mathbf{cp}}

\newcommand{\IR}{\mathbb{R}}

\newcommand{\Pp}{\mathfrak{P}}
\newcommand{\e}{\varepsilon}

\renewcommand{\phi}{\varphi}

\newcommand{\supp}{\mathrm{supp}}

\newcommand{\spn}{\mathrm{span}}

\input xy
\xyoption{all}

\title[Locally convex spaces and Schur type properties]{Locally convex spaces and Schur type properties}

\author[S.~Gabriyelyan]{Saak Gabriyelyan}
\address{Department of Mathematics, Ben-Gurion University of the
Negev, Beer-Sheva, P.O. 653, Israel}
\email{saak@math.bgu.ac.il}



\begin{document}

\begin{abstract}
In the main result of the paper we extend Rosenthal's characterization of Banach spaces with the Schur property by showing that for a quasi-complete locally convex space $E$ whose separable bounded sets are metrizable the following conditions are equivalent: (1) $E$ has the Schur property, (2) $E$ and $E_w$ have the same sequentially compact sets, where $E_w$ is  the space $E$ with the weak topology, (3) $E$ and $E_w$ have the same compact sets, (4) $E$ and $E_w$ have the same countably compact sets, (5) $E$ and $E_w$ have the same pseudocompact sets, (6) $E$ and $E_w$ have the same functionally bounded sets, (7) every bounded non-precompact sequence in $E$ has a subsequence which is equivalent to the unit  basis of $\ell_1$ and (8) every bounded non-precompact sequence in $E$ has a subsequence which is discrete and $C$-embedded in $E_w$.
\end{abstract}

\maketitle



\section{Introduction}


A locally convex space (lcs for short) $E$ is said to have the {\em  Schur property} if $E$ and $E_w$ have the same convergent sequences, where $E_w$ is  the space $E$ endowed with the weak topology $\sigma(E,E')$. By the classical Schur theorem the Banach space $\ell_1(\Gamma)$ has the Schur property for every set $\Gamma$.
The results of H.P.~Rosenthal in \cite{Rosenthal} (see also \S 2 of \cite{Bourgain-Ros}) show that a Banach space $E$ has the Schur property if and only if every $\delta$-separated sequence in the closed unit ball  of $E$ has a subsequence equivalent to the unit basis of $\ell_1$. Recall that a sequence $\{ x_n\}_{n\in\NN}$ is said to be $\delta$-separated if $\| x_i -x_j\|\geq \delta$ for all distinct $i,j\in\NN$. Since any bounded sequence in $E$ either is precompact or contains a $\delta$-separated subsequence for some $\delta>0$,  following \cite{HGM} the above characterization of Banach spaces with the Schur property can be formulated in a more convenient topological form.
\begin{theorem}[Rosenthal] \label{t:Rosenthal}
A Banach space $E$ has the Schur property if and only if every bounded non-precompact sequence in $E$ has a subsequence which is equivalent to the unit  basis of $\ell_1$.
\end{theorem}

Another characterization of Banach spaces with the Schur property was obtained recently by P.N. Dowling, D. Freeman, C.J. Lennard, E. Odell, B. Randrianantoanina and B. Turett in \cite{DowlingFree-1}: A Banach space $E$ has the Schur property if and only if every weakly compact subset of $E$ is contained in the closed convex hull of a weakly null sequence. A short proof of this result is given by W.B. Johnson, R. Lillemets and E. Oja in \cite{JLO}.
For other characterizations of Banach spaces with the Schur property see \cite{BFV,HGM,Pethe-Thakare}.

It is well known that a Banach space $E$ has the Schur property if and only if $E$ and $E_w$ have the same compact sets, and therefore, by  the Eberlein--\v{S}mulyan theorem, if and only if $E$ and $E_w$ have the same sequentially compact sets if and only if $E$ and $E_w$ have the same countably compact sets if and only if $E$ and $E_w$ have the same pseudocompact sets (the last equivalence follows from a result of V.~Pt\'{a}k, see \cite[\S~24.3(7)]{Kothe}). An analogous statement holds also for strict $(LF)$-spaces, see \cite[\S~24.3(9)]{Kothe}. Moreover, a remarkable result of M.~Valdivia \cite{Valdivia-77} (which states that $E_w$ is a $\mu$-space for every quasi-complete space $E$) easily implies the following assertion: If $E$ and $E_w$ have the same compact sets, then $E$ and $E_w$ have the same functionally bounded subsets. Note that in this assertion one cannot replace compact sets by convergent sequences even for complete spaces, see Proposition \ref{p:Schur-non-Glicksberg} below.

For a Tychonoff (=completely regular Hausdorff) space $X$ and a  property $\mathcal{P}$, denote by $\mathcal{P}(X)$ the set of all subspaces of $X$ with $\mathcal{P}$.
Analogously to the corresponding algebraic notion we shall say that a  locally convex space  $E$ {\em weakly respects $\mathcal{P}$} if $\mathcal{P}(E_w)=\mathcal{P}(E)$.
So every lcs $E$ weakly respects bounded sets, and the Schur property means that $E$ weakly respects the property $\mathcal{S}$ of being a convergent sequence. The following theorem extends Theorem \ref{t:Rosenthal} and is the main result of the paper (recall that an lcs $E$ is said to be {\em quasi-complete} if each closed bounded subset of $E$ is complete).

\begin{theorem} \label{t:respecting-quasi-complete-metriz}
Let $E$ be a quasi-complete lcs whose separable bounded sets are metrizable. Then  the following assertions are equivalent:
\begin{enumerate}
\item[{\rm (i)}] $E$ has the Schur property;
\item[{\rm (ii)}] $E$ weakly respects sequential compactness;
\item[{\rm (iii)}] $E$ weakly respects compactness;
\item[{\rm (iv)}] $E$ weakly respects countable compactness;
\item[{\rm (v)}] $E$ weakly respects pseudocompactness;
\item[{\rm (vi)}] $E$  weakly  respects functional boundedness;
\item[{\rm (vii)}]  every bounded non-precompact sequence in $E$ has a subsequence which is equivalent to the unit  basis of $\ell_1$;
\item[{\rm (viii)}] every bounded non-precompact sequence in $E$ has a subsequence which is discrete and $C$-embedded in $E_w$.
\end{enumerate}
If (i)-(viii) hold, then  every functionally bounded subset in $E_w$ is relatively compact in $E$.
\end{theorem}

Note that for any lcs $E$, (i) and (ii) are equivalent and conditions (iii)-(v)  imply the Schur property, see Proposition \ref{p:Schur=sequential-comp-LCS}. However, in general the Schur property does not imply the weak respecting compactness  (see \cite[Example~6 (p.~267)]{Wil} and  \cite[Example~19.19]{Dom} or the more general Proposition \ref{p:Schur-non-Glicksberg} below), and hence the condition of being metrizable for separable bounded sets in Theorem \ref{t:respecting-quasi-complete-metriz}  is essential.
We prove Theorem \ref{t:respecting-quasi-complete-metriz} in Section \ref{sec:Rosenthal} using (1) an extension of the Rosenthal $\ell_1$ theorem obtained recently by W.~Ruess \cite{ruess}, and (2) the aforementioned result of M.~Valdivia \cite{Valdivia-77}. As a corollary of Theorem \ref{t:respecting-quasi-complete-metriz} we provide numerous characterizations of the Schur property for strict $(LF)$-spaces, see  Corollary \ref{c:Schur-strict-LF}.

It is well known that every Banach space with the Schur property has the Dunford--Pettis property. Being motivated by this result and the results of  A.A. Albanese, J. Bonet and W.J. Ricker \cite{ABR} and J. Bonet and  M. Lindstr\"{o}m \cite{Bonet-Lin-93}, in the last section  we study convergent sequences in duals of locally convex spaces with the Schur property and Dunford--Pettis type properties and extend some results known for Banach spaces and Fr\'{e}chet spaces.


\section{Proof of Theorem \ref{t:respecting-quasi-complete-metriz}} \label{sec:Rosenthal}


A subset $A$ of a Tychonoff space  $X$ is called {\em functionally bounded in $X$} if every continuous real-valued function on $X$ is bounded on $A$, and $X$ is a {\em $\mu$-space} if every functionally bounded subset of $X$ has compact closure. Following \cite{Gab-Respected}, a topological space $X$ is a {\em countably $\mu$-space} if every countable functionally bounded  subset of $X$ has compact closure. Clearly, every $\mu$-space is a countably $\mu$-space, but the converse is not true  in general.

In what follows we consider the following families of compact type properties
\[
\mathfrak{P}_0 := \{ \mathcal{S}, \mathcal{C}, \mathcal{SC}, \mathcal{CC}, \mathcal{PC}\} \quad \mbox{ and } \quad \mathfrak{P} :=\mathfrak{P}_0 \cup \{\mathcal{FB}\},
\]
where $\mathcal{S}$, $\mathcal{C}$, $\mathcal{SC}$, $\mathcal{CC}$, $\mathcal{PC}$ or $\mathcal{FB}$  denote the property of being  a convergent sequence or being a compact, sequentially compact, countably compact, pseudocompact and functionally bounded subset of a topological space, respectively.

\begin{proposition} \label{pSubSchur-LCS}
Let $H$ be a subspace of an lcs $E$ and $\mathcal{P}\in \mathfrak{P}_0$. If $E$ weakly respects $\mathcal{P}$, then $H$ weakly respects $\mathcal{P}$ as well.
\end{proposition}

\begin{proof}
Let $K\in \mathcal{P}(H_w)$. Denote by $i: H \to E$ the identity map. Then $i$ is weakly continuous. So $K=i(K)\in \mathcal{P}(E_w)$  (note that if $K$ is pseudocompact, then $i(K)$ is also pseudocompact by \cite[3.10.24]{Eng}). Hence $K\in \mathcal{P}(E)$. Therefore $K\in \mathcal{P}(H)$. Thus $H$ weakly respects $\mathcal{P}$.  \qed
\end{proof}

It is clear that if an lcs $E$ weakly respect compactness then it has the Schur property. Next proposition shows that the Schur property is  weaker than other properties from $\Pp$.
\begin{proposition}  \label{p:Schur=sequential-comp-LCS}
Let $(E,\tau)$ be a locally convex space. Then:
\begin{enumerate}
\item[{\rm (i)}] $E$  has the Schur property if and only if it weakly respects sequential compactness;
\item[{\rm (ii)}] if $E$  weakly respects countable compactness, then $E$ has the Schur property;
\item[{\rm (iii)}] if $E$  weakly respects pseudocompactness, then $E$ has the Schur property;
\item[{\rm (iv)}] if $E$ is a countably $\mu$-space and  weakly  respects functional boundedness, then $E$ has the Schur property;
\item[{\rm (v)}] if every functionally bounded subset of $E_w$ has compact closure in $E$, then $E$ weakly respects all properties $\mathcal{P}\in\Pp$ and $E_w$ is a $\mu$-space.
\end{enumerate}
\end{proposition}

\begin{proof}
(i) Assume that $(E,\tau)$ has the Schur property and let $A$ be a sequentially compact subset of $E_w$. Take a sequence $S=\{ a_n\}_{n\in\NN}$ in $A$. Then $S$ has a weakly convergent subsequence $S'$. By the Schur property $S'$ converges in $\tau$. Hence $A$ is $\tau$-sequentially compact. Thus $E$ weakly respects sequential compactness.

Conversely, assume that $(E,\tau)$ weakly respects sequential compactness  and let $\{ a_n\}_{n\in\NN}$ be a sequence weakly converging to an element $a_0\in E$. Set $S:= \{ a_n\}_{n\in\NN} \cup\{ a_0\}$, so $S$ is weakly compact. Being countable $S$ is metrizable and hence weakly sequentially compact. So $S$ is sequentially compact in $\tau$. We show that $a_n\to a_0$ in $\tau$. Suppose for a contradiction that there is a $\tau$-neighborhood $U$ of $a_0$ which does not contain an infinite subsequence $S'$ of $S$. Then there is a subsequence   $\{ a_{n_k}\}_{k\in\NN}$ of $S'$ which $\tau$-converges to an element $g\in S$. Clearly, $g\not= a_0$ and $a_{n_k}\to g$ in the weak topology, and hence $a_n\not\to a_0$ in $\sigma(E,E')$, a contradiction. Therefore $a_n\to a_0$ in $\tau$. Thus $E$  has the Schur property.

(ii),(iii) Let $\{ a_n\}_{n\in\NN}$ be a sequence weakly converging to an element $a_0\in E$. Set $S:= \{ a_n\}_{n\in\NN} \cup\{ a_0\}$, so $S$ is weakly compact. Hence $S$ is $\sigma(E,E')$-countably compact. So $S$ is countably compact or pseudocompact in $\tau$, respectively. As any countable space is normal, in both cases $S$ is countably compact in $\tau$. We show that $a_n\to a_0$ in $\tau$. Suppose for a contradiction that there is a $\tau$-neighborhood $U$ of $a_0$ which does not contain an infinite subsequence $S'$ of $S$. Then $S'$ has a $\tau$-cluster point $g\in S$ and clearly $g\not= a_0$. Note that $g$ is also a cluster point of $S'$ in the weak topology. Hence $g=a_0$, a contradiction. Therefore $a_n\to a_0$ in $\tau$. Thus $E$  has the Schur property.

(iv) Let $S=\{ a_n: n\in\NN\} \cup \{ a_0\}$ be a sequence in $E$ which weakly converges to $a_0$. Since $S$ is also functionally bounded in $E_w$, we obtain that $S$ is closed and functionally bounded in $E$.  So $S$ is compact in $E$ because $E$ is a countably $\mu$-space. As the identity map $(S,\tau|_S)\to \big(S,\sigma(E,E')|_S\big)$ is a homeomorphism, $a_n\to a_0$ in $E$. Thus $E$ has the Schur property.

(v) Let $A\in\mathcal{P}(E_w)$. Then $A$ is functionally bounded in $E_w$. Therefore its $\tau$-closure $\overline{A}$ is compact in $E$, so  the identity map $id: \big(\overline{A}, \tau|_{\overline{A}}\big) \to \big(\overline{A}, \sigma(E,E')|_{\overline{A}}\big)$ is a homeomorphism. Hence $E_w$ is a $\mu$-space and $A\in \mathcal{P}(E)$.   Thus $E$ weakly respects $\mathcal{P}$.  \qed
\end{proof}

Recall that a locally convex space $E$ is called {\em semi-Montel} if every bounded subset of $E$ is relatively compact, and $E$ is a {\em Montel space} if it is a barrelled semi-Montel space.
\begin{proposition} \label{p:Schur-Montel-respected}
A semi-Montel space $E$ weakly respects all properties $\mathcal{P}\in\Pp$.
\end{proposition}

\begin{proof}
Let $A\in \mathcal{P}(E_w)$. Then $A$ is a functionally bounded subset of $E_w$ and hence $A$ is bounded in $E$. Therefore the closure $\overline{A}$ of $A$ in $E$ is compact and (v) of Proposition \ref{p:Schur=sequential-comp-LCS} applies.  \qed
\end{proof}

Recall that a topological group $G$ is called {\em sequentially complete} if every Cauchy sequence has a limit point. A locally convex space $E$ is called {\em weakly sequentially complete} if the space $E_w$ is sequentially complete. The following assertion is well known for Banach spaces, we give its proof only for the sake of completeness and the reader convenience.
\begin{proposition} \label{p:weak-sequential-complete-LCS}
Let $(E,\tau)$ be an lcs  with the Schur property. Then $E$ is sequentially complete if and only if it   is weakly sequentially complete.
\end{proposition}

\begin{proof}
Assume that $E$ is sequentially complete.  Let $\{ g_n\}_{n\in\NN}$ be a Cauchy sequence in $E_w$. Then for any two strictly increasing sequences $\{ m_i\}_{i\in\NN}$ and $\{ k_i\}_{i\in\NN} $ in $\NN$, the sequence $\{ g_{m_i} -g_{k_i} \}_{i\in\NN}$  weakly converges to $0\in E$. By the Schur property $\{ g_{m_i} -g_{k_i} \}_{i\in\NN}$ converges to zero in $\tau$. So $\{ g_n\}_{n\in\NN}$ is a Cauchy sequence in $E$ and hence it converges to some $g\in E$. Therefore $g_n\to g$ in $\sigma(E,E')$. Thus $E$ is weakly sequentially complete.

Conversely, let $E$ be weakly sequentially complete.  If $\{ g_n\}_{n\in\NN}$ is a Cauchy sequence in $E$, then it is a Cauchy sequence in $\sigma(E,E')$ and hence weakly converges to some $g\in E$. By the Schur property,  $g_n\to g$ in the original topology $\tau$. Thus $E$ is sequentially complete.  \qed
\end{proof}

For a nonzero $x=(x_n)\in\ell_\infty$, we denote by $\supp(x):=\{ n\in\NN: x_n\not= 0\}$ the {\em support} of $x$. For every $n\in\NN$, set $e_n :=(0,\dots,0,1,0,\dots)$, where $1$ is placed in position $n$. We shall consider the sequence $\{ e_n: n\in\NN\}$ also as the {\em (standard) unit basis}  of $\ell_1$.  The span of a subset $A$ of an lcs $E$ is denoted by $\spn(A)$.

We shall say that a sequence $A=\{ a_n\}_{n\in\NN}$ of an lcs $E$ is {\em equivalent to the unit basis $\{ e_n: n\in\NN\}$ of $\ell_1$} if there exists a linear topological isomorphism $R$ from the closure of $\spn(A)$ onto a subspace of $\ell_1$ such that $R(a_n)=e_n$ for every $n\in\NN$ (we do not assume that the closure of $\spn(A)$ is complete).

\begin{lemma} \label{l:l1-sequence}
Let $A=\{ a_n\}_{n\in\NN}$ be a sequence in an lcs $E$ which is equivalent to the unit basis of $\ell_1$. Then $A$ is not weakly functionally bounded.
\end{lemma}

\begin{proof}
It is well known that there exists an isometric isomorphism $T$ from $\ell_1$ into $\ell_\infty$. We claim that the sequence $S:=\{ T(e_n): n\in\NN\}$ is not weakly functionally bounded in $\ell_\infty$. Indeed, suppose for a contradiction that $S$ is weakly functionally bounded in $\ell_\infty$. Since $(\ell_\infty)_w$ is a $\mu$-space by \cite{Valdivia-77}, the weak closure  $K$ of $S$ is a weakly compact subset of $\ell_\infty$. As $(\ell_1)_w$ is a closed subspace of $(\ell_\infty)_w$, we obtain that $K$ is compact in $(\ell_1)_w$. Therefore $K$ is norm compact by the Schur property of $\ell_1$, a contradiction. Thus $S$ is not weakly functionally bounded in $\ell_\infty$.

Now let $M$ be the closure of $\spn(A)$ in $ E$ and let $R$ be a linear topological isomorphism of $M$ onto a subspace of $\ell_1$ such that $R(a_n)=e_n$ for every $n\in\NN$. The injectivity of $\ell_\infty$ (see Theorem 10.1.2 of \cite{NaB}) implies that the operator $T\circ R$ can be extended to a continuous operator $Q$ from $E$ to $\ell_\infty$. Since $Q$ is also weakly continuous, the claim implies that $A$ is not weakly functionally bounded in $E$.  \qed
\end{proof}

Following \cite{GKKLP}, a locally convex space $E$ is said to have the {\em Rosenthal property} if every bounded sequence in $E$ has a subsequence which either (1) is Cauchy in the weak topology, or (2) is equivalent to the unit basis of $\ell_1$.

\begin{proposition} \label{p:Rosenthal-Schur}
Let $(E,\tau)$ be a  sequentially complete lcs with the Rosenthal property. If $E$ has the Schur property, then every functionally bounded subset $A$ of $E_w$ is relatively sequentially compact in $E$.
\end{proposition}

\begin{proof}
We have to show that every sequence $\{ a_n\}_{n\in\NN} \subseteq A$ has a convergent subsequence in $E$. 
Clearly, $A$ is a bounded subset of $E$. Hence, by the  Rosenthal property, $\{ a_n\}_{n\in\NN}$ has a subsequence $\{ b_n\}_{n\in\NN}$ which either (1) is Cauchy in the weak topology $\tau_w$, or (2) is equivalent to the unit basis of $\ell_1$. The case (2) is not fulfilled by Lemma \ref{l:l1-sequence}, so $\{ b_n\}_{n\in\NN}$  is weakly Cauchy. Then, by Proposition \ref{p:weak-sequential-complete-LCS}, $b_n$  weakly converges to a point $b\in E$, and hence,  by the Schur property, $b_n\to b$ in the original topology $\tau$. Thus $A$ is relatively sequentially compact in $E$.  \qed
\end{proof}

Taking into account that every Banach space is sequentially complete 
and being motivated by the Rosenthal property and  the Eberlein--\v{S}mulyan theorem, we denote by $\mathbf{RES}$ the class of all sequentially complete locally convex spaces $E$ with the Rosenthal property such that the closure $\overline{A}$ of every relatively  sequentially compact subset $A$ of $E$ is compact.

\begin{theorem} \label{t:Rosenthal-Ruess-Schur-basic}
If $(E,\tau)\in \mathbf{RES}$, then the following assertions are equivalent:
\begin{enumerate}
\item[{\rm (i)}] $E$ has the Schur property;
\item[{\rm (ii)}] there is $\mathcal{P}\in \Pp_0$ such that $E$ weakly respects $\mathcal{P}$;
\item[{\rm (iii)}] $E$ weakly respects $\mathcal{FB}$ and is a countably $\mu$-space;
\item[{\rm (iv)}] $E$ weakly respects all  $\mathcal{P}\in \Pp$ and is a countably $\mu$-space;
\item[{\rm (v)}] $E$ is a countably $\mu$-space and every bounded non-functionally bounded sequence in $E$ has a subsequence which is equivalent to the unit  basis of $\ell_1$.
\end{enumerate}
In these cases every weakly functionally bounded subset of $E$ is relatively compact in $E$ and the space $E_w$ is a $\mu$-space.
\end{theorem}

\begin{proof}
(i)$\Rightarrow$(iv) Let $A\in\mathcal{P}(E_w)$ for some $\mathcal{P}\in\Pp$. Then $A$ is a functionally bounded subset of $E_w$. By Proposition \ref{p:Rosenthal-Schur}, $A$ is relatively sequentially compact in $E$. As $E\in \mathbf{RES}$ we obtain that the $\tau$-closure $\overline{A}$ of $A$ is $\tau$-compact, i.e., $A$ is relatively compact in $E$.  Therefore the space $E_w$ and hence also $E$ are  $\mu$-spaces, and $A\in \mathcal{FB}(E)$. Thus $E$ weakly respects functional boundedness. If $\mathcal{P}\in\Pp_0$ and $\overline{A}$  is $\tau$-compact,   the identity map $(\overline{A}, \tau|_{\overline{A}}) \to (\overline{A}, \tau_w|_{\overline{A}})$ is a homeomorphism and hence  $A\in\mathcal{P}(E)$. Thus $E$ weakly respects $\mathcal{P}$.

(ii)$\Rightarrow$(i) follows from Proposition  \ref{p:Schur=sequential-comp-LCS}. The implications (i)$\Rightarrow$(ii), (iv)$\Rightarrow$(i) and (iv)$\Rightarrow$(iii) are clear.

(iii)$\Rightarrow$(i) Let $S$ be a convergent sequence in $E_w$ with the limit point. Then $S\in \mathcal{FB}(E_w)$, and hence $S$ is functionally bounded in $E$. As $E$ is a countably $\mu$-space, the closed countable subset $S$ of $E$ is compact. Since the identity map $(S,\tau|_S) \to (S, \tau_w|_S)$ is a homeomorphism, we obtain that $S$ is a convergent sequence in $E$. Thus $E$ has the Schur property.

(i),(iii)$\Rightarrow$(v)   Let $S=\{ g_n\}_{n\in\NN}$ be a  bounded non-functionally bounded sequence in $E$. Observe that $S$ is also not functionally bounded in $E_w$ since $E$ weakly respects functional boundedness. Take a continuous real-valued function $f$ on $E_w$ which is unbounded on $S$ and take a subsequence $S'=\{ g_{n_k}\}_{k\in\NN}$ of $S$ such that $f(g_{n_k})\to\infty$.
We claim that $S'$ does not have a weakly Cauchy subsequence $S''$. Indeed, otherwise, the weak sequential completeness of $E$ (see Proposition \ref{p:weak-sequential-complete-LCS}) would imply that $S''$ converges in $\tau_w$. Hence $f$ is bounded on $S''$ that is impossible since $f(g_{n_k})\to\infty$.
Finally, the Rosenthal property implies that $S$ has a subsequence which is equivalent to the unit  basis of $\ell_1$.

(v)$\Rightarrow$(i) Let $S$ be a weakly convergent sequence in $E$ with the limit point. Then $S$ is bounded. We claim that $S$ is functionally bounded in $E$. Indeed, otherwise, $S$ would contain a subsequence $S'$  which is equivalent to the unit  basis of $\ell_1$. Therefore $S$ is not weakly functionally bounded by Lemma \ref{l:l1-sequence}, a contradiction. Now since $E$ is a countably $\mu$-space, the closed countable subset $S$ of $E$ is compact. As the identity map $(S,\tau|_S) \to (S, \tau_w|_S)$ is a homeomorphism, we obtain that $S$ is a convergent sequence in $E$. Thus $E$ has the Schur property.  \qed
\end{proof}

Being motivated by a result of S.~Hern\'{a}ndez and S.~Macario \cite[Theorem~3.2]{HM}, we prove the following sufficiently general result. Recall that a subset $A$ of a topological space $(X,\tau)$ is called {\em $C$-embedded} ({\em $C^\ast$-embedded}) if every continuous (respectively,  bounded) real-valued function $f$ on $(A,\tau|_A)$ can be extended to a continuous function $\overline{f}$ on $X$.

\begin{theorem} \label{t:Glicksberg-respecting-lcs}
If  an lcs $(E,\tau)$ is such that $E_w$ is a $\mu$-space, then the following assertions are equivalent:
\begin{enumerate}
\item[{\rm (i)}] $E$ weakly respects compactness;
\item[{\rm (ii)}] $E$ weakly respects countable compactness and $E$  is a $\mu$-space;
\item[{\rm (iii)}] $E$ weakly respects pseudocompactness and $E$  is a $\mu$-space;
\item[{\rm (iv)}] $E$  weakly  respects functional boundedness and $E$  is a $\mu$-space;
\item[{\rm (v)}] $E$  is a $\mu$-space and every non-functionally bounded subset $A$ of $E$ has an infinite subset $B$ which is discrete and $C$-embedded in $E_w$.
\end{enumerate}
If (i)-(v) hold, then  every functionally bounded subset in $E_w$ is relatively compact in $E$.
\end{theorem}

\begin{proof}
(i)$\Rightarrow$(ii) Let $A$ be a  countably compact subset of $E_w$.  As $E_w$ is a $\mu$-space, the $\tau_w$-closure $\overline{A}$ of $A$ is compact in $E_w$. Therefore $\overline{A}$ is compact in $E$, and hence $A$ is relatively compact in $E$. Since the identity map $(\overline{A},\tau|_{\overline{A}})\to (\overline{A},\tau_w|_{\overline{A}})$ is a homeomorphism, we obtain that $A$ is  countably compact in $E$. Thus $E$ respects countable compactness. The same proof shows that  every functionally bounded subset in $E_w$ is relatively compact in $E$, and in particular $E$ is a $\mu$-space.

(ii)$\Rightarrow$(iii) Let $A$ be a pseudocompact subset of $E_w$. Then the closure $K$ of $A$ in $E_w$ is $\tau_w$-compact because $E_w$ is a $\mu$-space. Therefore $K$ is countably compact in $E$. Being closed $K$ also is compact in $E$ since $E$ is a $\mu$-space. As the identity map $(K,\tau|_{K})\to (K,\tau_w|_{K})$ is a homeomorphism, we obtain that $A$ is pseudocompact in $E$. Thus $E$ weakly respects pseudocompactness.

The implication (iii)$\Rightarrow$(iv) is proved analogously to (ii)$\Rightarrow$(iii).

(iv)$\Rightarrow$(v) Let $A$ be a non-functionally bounded  subset of $E$. As $E$ weakly respects functional boundedness it follows that  $A$ is not functionally bounded in $E_w$. Let $f$ be a continuous real-valued function on $E_w$ which is unbounded on $A$. If we take $B$ as a sequence $\{ a_{n}\}_{n\in\NN}$ in $A$ such that $|f(a_{n+1})|> |f(a_{n})| +1$ for all $n\in\NN$, then $B$ is discrete and $C$-embedded in $E_w$.

(v)$\Rightarrow$(i) Let $K$ be a compact subset of $E_w$. Then $K$ must be functionally bounded in $E$. Since $E$ is a $\mu$-space and $K$ is also closed in $E$ we obtain that $K$ is compact in $E$. Thus $E$ weakly respects compactness.  \qed
\end{proof}

We need also the following mild completeness type property. An   lcs  $E$ is said to have the {\em $\CP$-property} if every countable precompact subset of $E$ has compact closure. Clearly, any complete lcs has the $\CP$-property, and each lcs with the $\CP$-property is a countably $\mu$-space.
\begin{lemma} \label{l:func-bounded-precompact}
Let $E$ be a locally convex space. Then:
\begin{enumerate}
\item[{\rm (i)}] every functionally bounded subset $A$ of $E$ is precompact;
\item[{\rm (ii)}] if $E$ has the $\CP$-property, then a separable subset $B$ of $E$ is functionally bounded if and only if $B$ is precompact.
\end{enumerate}
\end{lemma}

\begin{proof}
(i) If $A$ is not precompact, Theorem 5 of \cite{BGP} implies that $A$ has an infinite uniformly discrete subset $C$, i.e., there is a neighborhood $U$ of zero in $E$ such that $c-c' \not\in U$ for every distinct $c,c'\in C$. Thus,  by Lemma 2.1 of \cite{Gab-Top-Nul}, $C$ and hence also $A$ are not functionally bounded, a contradiction.

(ii) follows from (i) and the $\CP$-property.  \qed
\end{proof}

In \cite[Lemma~3]{Diaz-89} J.C.~D\'{\i}az extends the Rosenthal $\ell_1$ theorem to all Fr\'{e}chet spaces. A much more general result was obtained recently by W.~Ruess in \cite{ruess}. Recall that an lcs $E$ is locally complete if every closed disc in $E$ is a Banach disc; every sequentially complete lcs is locally complete by Corollary 5.1.8 of \cite{bonet}.
\begin{theorem}[Ruess] \label{ruess}
Every locally complete lcs $E$ whose every separable bounded set is metrizable has the Rosenthal property.
\end{theorem}

Below we give some examples of locally convex spaces which belong to the class  $\mathbf{RES}$ and have the $\CP$-property.
An lcs $(E,\tau)$ is called an {\em $(LM)$-space} and write $E=\underrightarrow{\lim}\, E_n$ if there is a sequence $\{ (E_n,\tau_n)\}_{n\in\NN}$ of metrizable locally convex spaces such that $(E_n,\tau_n)$ is continuously included in $(E_{n+1},\tau_{n+1})$, and $\tau$ is the finest Hausdorff locally convex topology on $E=\bigcup_n E_n$ such that $(E_n,\tau_n)$ is continuously included in $(E,\tau)$. If in addition all the spaces $(E_n,\tau_n)$ are Fr\'{e}chet spaces and $\tau_{n+1} |_{E_n} = \tau_n$ holds for every $n\in\NN$, the space $E$ is called a {\em strict $(LF)$-space}.

\begin{proposition} \label{p:Schur-Ruess}
A locally convex space $(E,\tau)$ belongs to $\mathbf{RES}$ and has the $\CP$-property if one of the following conditions holds:
\begin{enumerate}
\item[{\rm (i)}] $E$ is  quasi-complete and every separable bounded subset of $E$ is metrizable;
\item[{\rm (ii)}] $E$ is a strict $(LF)$-space (in particular, a Fr\'{e}chet space);
\item[{\rm (iii)}] $E$ is the strong dual $H'_\beta$ of a quasinormable metrizable lcs $H$.
\end{enumerate}
\end{proposition}

\begin{proof}
(i) Let $A$ be a precompact subset of $E$. Then $A$ is bounded, and since $E$ is  quasi-complete, the closure $\overline{A}$ of $A$  is complete. Therefore $\overline{A}$ is compact by Theorem 3.4.1 of \cite{NaB}. Thus $E$ has the $\CP$-property. The space $E$ has the Rosenthal property by Theorem \ref{ruess}. So to show that $E$ belongs to the class  $\mathbf{RES}$ we have to prove that the closure $\overline{B}$ of a relatively sequentially compact subset $B$ of $E$ is compact. It is easy to see that $B$ is functionally bounded in $E$. Therefore $B$ is precompact by (i) of Lemma \ref{l:func-bounded-precompact}. As we proved above,  $\overline{B}$ is compact.

(ii) Let $E=\underrightarrow{\lim}\, E_n$, where all the $E_n$ are Fr\'{e}chet spaces. Note that, for every $n\in\NN$, the space $E_n$ is a closed subspace of $E$. Theorem 12.1.10 of \cite{NaB} implies that $E$ is complete. By Theorem 12.1.7 of \cite{NaB}, any bounded subset of $E$ sits in some $E_n$ and hence is metrizable. Now (i) applies.

(iii) The space $E$ is quasi-complete by Proposition 11.2.3 of \cite{Jar}, and Theorem 2 (see also the diagram before this theorem) of \cite{bierstedt1} implies that every bounded subset of $E$ is metrizable. Now (i) applies.  \qed
\end{proof}

Now we are ready to prove the main result of the paper.

\medskip
{\em Proof of Theorem \ref{t:respecting-quasi-complete-metriz}.}
As we showed in the proof of (i) in Proposition \ref{p:Schur-Ruess}, the quasi-complete space $E$ is a $\mu$-space and has the $\CP$-property. Also the space $E_w$ is a $\mu$-space by  \cite{Valdivia-77}. Now the theorem follows from Theorems \ref{t:Rosenthal-Ruess-Schur-basic} and  \ref{t:Glicksberg-respecting-lcs} and Lemma \ref{l:func-bounded-precompact}.  \qed

\medskip

Let $X$ and $Y$ be Tychonoff spaces. The space $C(X,Y)$ of all continuous functions from $X$ to $Y$ endowed with the compact-open topology $\tau_k$ is denoted by $\CC(X,Y)$. If $Y=\IR$, we set $C(X):=C(X,\IR)$ and $\CC(X):=\CC(X,\IR)$. The sets
\[
[K;\e] :=\{ f\in C(X): |f(x)|<\e \; \forall x\in K\},
\]
where $K$ is a compact subset of $X$ and $\e>0$, form a base at zero of $\tau_k$.

Let $E$ be an lcs over the field $\mathbf{F}$ of real numbers $\IR$ or complex numbers $\mathbb{C}$. We denote by $\mu(E,E')$ the Mackey topology on $E$ and set $E'_\beta :=(E',\beta(E',E))$, where $\beta(E',E)$ is the strong topology on $E'$. If $A$ is a subset of $E$, we denote by $A^\circ$ and $\overline{\mathrm{acx}}(A)$ the polar of $A$ and the closed absolutely convex hull of $A$, respectively. The Krein theorem \cite[\S~24.5(4)]{Kothe} states that if $K$ is a weakly compact subset of $E$, then $\overline{\mathrm{acx}}(K)$ is weakly compact if and only if $\overline{\mathrm{acx}}(K)$ is $\mu(E,E')$-complete. We shall say that $E$ has the {\em Krein property} or is a {\em Krein space} if $\overline{\mathrm{acx}}(K)$ is weakly compact  for every weakly compact subset $K$ of $E$. Therefore $E$ is a Krein space if $(E,\mu(E,E'))$ is quasi-complete. In particular, every quasibarrelled quasi-complete space is a Krein space. By the definition of $\mu(E',E)$, the Mackey topology $\mu(E',E)$ on $E'$ is always weaker than the restriction $\tau_k|_{E'}$ of the compact-open topology $\tau_k$ of $\CC(E_w, \mathbf{F})$ onto $E'$. Below we show that the equality $\mu(E',E)=\tau_k|_{E'}$ characterizes Krein spaces.

\begin{proposition} \label{p:Krein-characterization}
Let $E$ be an lcs over $\mathbf{F}=\IR$ or $\mathbb{C}$. Then:
\begin{enumerate}
\item[{\rm (i)}] $E$ is a Krein space if and only if $(E',\mu(E',E))$ is a subspace of $\CC(E_w, \mathbf{F})$;
\item[{\rm (ii)}] if $E$ is quasi-complete and weakly respects compactness, then   $(E',\mu(E',E))$ is a subspace of $\CC(E, \mathbf{F})$;
\item[{\rm (iii)}] if $E$ is a quasi-complete Mackey space, then $E$ weakly respects compactness if and only if  $(E',\mu(E',E))$ is a subspace of $\CC(E, \mathbf{F})$.
\end{enumerate}
\end{proposition}

\begin{proof}
(i) Assume that $E$ is a Krein space.  Fix a weakly compact subset $K$  of $E$. Then $\overline{\mathrm{acx}}(K)$ is also weakly compact by the Krein property. Therefore $\mu(E',E)\geq\tau_k|_{E'}$. Thus $\mu(E',E)=\tau_k|_{E'}$ and hence  $(E',\mu(E',E))$ is a subspace of $\CC(E_w, \mathbf{F})$. Conversely, assume that  $(E',\mu(E',E))$ is a subspace of $\CC(E_w, \mathbf{F})$. Fix a  weakly compact subset $K$ of $E$. Then there exists an absolutely convex weakly compact subset $C$ of $E$ such that $C^\circ \subseteq K^\circ$. Therefore $K\subseteq K^{\circ\circ} \subseteq C^{\circ\circ}=C$. Thus $\overline{\mathrm{acx}}(K)$ is a weakly compact subset of $E$.

(ii) If $K$ is an absolutely convex weakly compact subset of $E$, then $K$ is compact in $E$. Therefore $\mu(E',E)\leq \tau_k|_{E'}$. Conversely, for every compact subset $K$ of $E$, the quasi-completeness of $E$ implies that the closed absolutely convex  hull of $K$ is also compact in $E$. Thus  $\mu(E',E)\geq \tau_k|_{E'}$ and hence $\mu(E',E)=\tau_k|_{E'}$.

(iii) The necessity follows from (ii). To prove sufficiency, let $K$ be a weakly compact subset of $E$. As $E$ is a Krein space, the set $K_0:=\overline{\mathrm{acx}}(K)$ is also weakly compact. By assumption there is a compact subset $C_0$ of $E$ such that $C_0^\circ \subseteq K_0^\circ$. Since $E$ is quasi-complete, we obtain that the set $C:=\overline{\mathrm{acx}}(C_0)$ is also a compact subset of $E$. Then $K\subseteq K^{\circ\circ} \subseteq C_0^{\circ\circ} \subseteq C^{\circ\circ}=C$. Therefore $K$ being closed in $E$ is a compact subset of $E$. Thus $E$ weakly respects compactness.  \qed
\end{proof}

The class of strict $(LF)$-spaces is one of the most important classes of locally convex spaces. Below we apply Theorem \ref{t:respecting-quasi-complete-metriz} and Proposition \ref{p:Krein-characterization} to characterize strict $(LF)$-spaces with the Schur property.

\begin{corollary} \label{c:Schur-strict-LF}
Let $E=\underrightarrow{\lim}\, E_n$ be a strict $(LF)$-space over  $\mathbf{F}=\IR$ or $\mathbb{C}$. Then the following assertions are equivalent:
\begin{enumerate}
\item[{\rm (i)}] there is a $\mathcal{P}\in \Pp$ such that $E$ weakly respects $\mathcal{P}$;
\item[{\rm (ii)}] $E$ weakly respects all  $\mathcal{P}\in\Pp$;
\item[{\rm (iii)}] every bounded non-precompact sequence in $E$ has a subsequence which is equivalent to the unit  basis of $\ell_1$;
\item[{\rm (iv)}] there is a $\mathcal{P}\in \Pp$ such that all the spaces $E_n$ weakly respect $\mathcal{P}$;
\item[{\rm (v)}]  for every $n\in\NN$, the space $E_n$ weakly respects all $\mathcal{P}\in\Pp$;
\item[{\rm (vi)}] every non-precompact bounded subset of $E$ has an infinite subset which is discrete and $C$-embedded in $E_w$;
\item[{\rm (vii)}] $(E',\mu(E',E))$ is a subspace of $\CC(E, \mathbf{F})$.
\end{enumerate}
\end{corollary}

\begin{proof}
The equivalences (i)$\Leftrightarrow$(ii)$\Leftrightarrow$(iii)$\Leftrightarrow$(vi) follow from Theorem \ref{t:respecting-quasi-complete-metriz} and (ii) of Proposition \ref{p:Schur-Ruess}.

(iii)$\Rightarrow$(v) Fix $n\in\NN$ and let $S$ be a bounded non-precompact sequence in $E_n$. As $E_n$ is a closed subspace of $E$, there is a subsequence $S'$ of  $S$  which is equivalent to the unit basis of $\ell_1$. Clearly, the closure of $\spn(S')$ in $E$ is contained in $E_n$. Therefore $E_n$ respects all the properties $\mathcal{P}\in\Pp$ by Theorem  \ref{t:respecting-quasi-complete-metriz} and (ii) of Proposition \ref{p:Schur-Ruess}.

(v)$\Rightarrow$(iv) is trivial. To prove the implication (iv)$\Rightarrow$(iii), let $S$ be a bounded non-precompact sequence in $E$. Then $S\subseteq E_n$ for some $n\in\NN$. Applying Theorem  \ref{t:respecting-quasi-complete-metriz} and (ii) of Proposition \ref{p:Schur-Ruess} to the Fr\'{e}chet space $E_n$, we obtain that $S$ has a subsequence $S'$ which is equivalent to the unit basis of $\ell_1$. It remains to note that $E_n$ is a closed subspace of $E$.

Taking into account that $E$ is a complete Mackey space (see \cite[Theorem~12.1.10]{NaB} and \cite[Corollary~8.8.11]{Jar}), (ii) and (vii) are equivalent by (iii) of  Proposition \ref{p:Krein-characterization} and Theorem \ref{t:respecting-quasi-complete-metriz}.  \qed
\end{proof}

Let $K$ be an infinite compact space. Then the Banach space $C(K)$ does not have the Schur property since it contains an isomorphic copy of $c_0$, see Theorem 14.26 of \cite{fabian-10}. The next proposition (which, perhaps,  is known but hard to find explicitly stated) generalizes this result.

\begin{proposition} \label{p:Schur-Ck}
For a Tychonoff space $X$, the space $\CC(X)$ has the Schur property if and only if $X$ does not contain an infinite compact subset.
\end{proposition}

\begin{proof}
Let $\CC(X)$ have the Schur property. Suppose for a contradiction that $X$ contains an infinite compact subset $K$. By Lemma 11.7.1 of \cite{Jar}, take a one-to-one sequence $\{x_n\}_{n\in\NN}$ in $K$ and a sequence $\{U_n\}_{n\in\NN}$ of open subsets of $X$ such that $x_n\in U_n$ for each $n\in\NN$ and  $\overline{U_n}\cap \overline{U_m} =\emptyset$ for every distinct $n,m\in\NN$. For  every $n\in\NN$, take a continuous function $f_n:X\to [0,1]$ with support in $U_n$ and such that $f_n(x_n)=1$. We claim that $f_n$ weakly converges to zero. Indeed, fix a regular Borel measure $\mu\in M_c(X)=\CC(X)'$ with compact support. We can assume that $\mu$ is positive. Since $\mu$ is finite and regular, observe that $\mu(U_n)\to 0$. Then $0\leq \mu(f_n) \leq \mu(U_n) \to 0$.
This proves the claim. On the other hand, by construction we have $f_n \not\in [K;1/3]$ for every $n\in\NN$. Therefore $f_n\not\to 0$ in the compact-open topology. Thus $\CC(X)$ does not have the Schur property.

Conversely, if $X$ does not contain  infinite compact subsets, then the compact-open topology $\tau_k$ coincides with the pointwise topology and hence $\big(\CC(X)\big)_w= \CC(X)$. Thus $\CC(X)$ trivially has the Schur property.  \qed
\end{proof}


\section{The Schur property and Dunford--Pettis type properties} \label{sec:Schur-DP}


In this section  we give concrete constructions of Schur spaces and 
extend some known results for Banach and Fr\'{e}chet spaces to larger classes of locally convex spaces.

Recall that a Tychonoff space $X$ is called an {\em angelic space} if (1) every relatively countably compact subset of $X$ is relatively compact, and (2) any compact subspace of $X$ is Fr\'{e}chet--Urysohn. Note that any subspace of an angelic space is angelic, and a subset $A$ of an angelic space $X$ is compact if and only if it is countably compact if and only if $A$ is sequentially compact, see Lemma 0.3 of \cite{Pryce}. Being motivated by the last property we say that a  Tychonoff space $X$ is {\em sequentially angelic} if a subset $K$ of $X$ is compact if and only if $K$ is sequentially compact. 

Recall that an lcs $E$ has the {\em Dunford--Pettis property} ($(DP)$ {\em property} for short)  if every absolutely convex $\sigma(E,E')$-compact subset of $E$ is precompact for the topology $\tau_{\Sigma'}$ of uniform convergence on the absolutely convex, equicontinuous, $\sigma(E',E'')$-compact subset of $E'$ (see \S 9.4 of \cite{Edwards}); $E$ has the {\em Grothendieck property} if every weak-$\ast$ convergent sequence in $E'$ is weakly convergent. Recall also (see \cite{BFV}) that a Banach space $E$ has  the {\em $\ast$-Dunford--Pettis property} ($\ast$-$(DP)$ {\em property}) if given a weakly null sequence $\{ x_n\}_{n\in\NN}$ in $E$ and a weakly-$\ast$ null sequence  $\{ \chi_n\}_{n\in\NN}$ in $E'$, then $\lim_n \chi_n(x_n)=0$. Analogously we say that an lcs $E$ has
\begin{enumerate}
\item[$\bullet$] the {\em sequential Dunford--Pettis property} ($(sDP)$ {\em property})  if  given weakly null sequences $\{ x_n\}_{n\in\NN}$ and $\{ \chi_n\}_{n\in\NN}$ in $E$ and $E'_\beta$, respectively, then $\lim_n \chi_n(x_n)=0$;

\item[$\bullet$] the {\em $\ast$-sequential Dunford--Pettis property} ($\ast$-$(sDP)$ {\em property}) if given a weakly null sequence $\{ x_n\}_{n\in\NN}$ in $E$ and a weakly-$\ast$ null sequence  $\{ \chi_n\}_{n\in\NN}$ in $E'$, then $\lim_n \chi_n(x_n)=0$.
\end{enumerate}
Clearly, the $\ast$-$(sDP)$ property implies the $(sDP)$ property, but the converse is not true in general as the Banach space $c_0$ shows.
Any Banach space $E$ with the $\ast$-$(DP)$ property  contains an isomorphic copy of $\ell_1$, see \cite[Proposition~4]{JPZ}. It is easy to see that every Banach space $E$ with the Schur property has the $\ast$-$(sDP)$ property, but in general the converse is false (for example $E=\ell_\infty$).
Proposition 3.3 of \cite{ABR} implies  that: (1) every barrelled space with the $(DP)$-property has the $(sDP)$ property, and (2) if both $E$ and $E'_\beta$ are sequentially angelic and $E$ has the $(sDP)$ property, then $E$ has the $(DP)$ property. Moreover, if $E$ is a strict $(LF)$-space, then $E$ has the $(DP)$ property if and only if it has the $(sDP)$-property, see Corollary 3.4 of \cite{ABR}. Banach spaces with the  $\ast$-$(DP)$ property are studied in \cite{CGL-08}.

The next proposition generalizes the corresponding well known results for Banach spaces. Recall that an lcs $E$ is called {\em $c_0$-barrelled} if every $\sigma(E',E)$-null sequence in $E'$ is equicontinuous.

\begin{proposition} \label{p:Schur-Dunford-Pettis}
Let $E$ be a locally convex space. Then:
\begin{enumerate}
\item[{\rm (i)}] if $E$ is $c_0$-barrelled and  has the Schur property, then $E$ has the $\ast$-$(sDP)$ property;
\item[{\rm (ii)}] if $E'_\beta$ has the Schur property, then $E$ has the $(sDP)$ property;
\item[{\rm (iii)}] if $E$ is quasibarrelled and $E'_\beta$ has the $(sDP)$ property, then also $E$ has the $(sDP)$ property.
\end{enumerate}
\end{proposition}
\begin{proof}
(i) Let $\{ x_n\}_{n\in\NN}$ be a weakly null sequence in $E$ and $\{ \chi_n\}_{n\in\NN}$ be a weakly-$\ast$ null sequence  in $E'_\beta$. As $E$ is $c_0$-barrelled,  there is a neighborhood $U$ of zero in $E$ such that $\{ \chi_n\}_{n\in\NN} \subseteq U^\circ$. By the Schur property, for each $\e>0$ choose $N\in \NN$ such that $x_n \in \e U$ for every $n>N$. Then $|\chi_n(x_n)|\leq \e$ for $n>N$. Thus $\chi_n(x_n)\to 0$ and $E$ has the $\ast$-$(sDP)$ property.

(ii) Let $\{ x_n\}_{n\in\NN}$ and $\{ \chi_n\}_{n\in\NN}$ be weakly null sequences in $E$ and $E'_\beta$, respectively. Since $E'_\beta$ has the Schur property, $\chi_n \to 0$ in the strong topology $\beta(E',E)$ on $E'$. Clearly, $S=\{ x_n\}_{n\in\NN}$ is a bounded subset of $E$. Therefore for each $\e>0$ there is an $N\in\NN$ such that $|\chi_n(x)|<\e$ for every $x\in S$ and every $n>N$. Thus $\lim_n \chi_n(x_n)=0$.

(iii) Let $\{ x_n\}_{n\in\NN}$ and $\{ \chi_n\}_{n\in\NN}$ be weakly null sequences in $E$ and $E'_\beta$, respectively. Since $E$ is quasibarrelled, it is a subspace of $(E'_\beta)'_\beta$ by Theorem 15.2.3 of \cite{NaB}. So $x_n\to 0$ also in $\sigma(E'',E''')$. Now the $(sDP)$ property of $E'_\beta$ implies $\chi_n (x_n) \to 0$. Thus $E$ has the $(sDP)$ property.  \qed
\end{proof}

\begin{remark} {\em
P.~Pethe and N.~Thakare showed in \cite{Pethe-Thakare} that the Banach dual  $E'$ of a Banach space $E$ has the Schur property if and only if $E$ has the Dunford--Pettis property and does not contain an isomorphic copy of $\ell_1$.   \qed }
\end{remark}

\begin{proposition} \label{p:Dunford-Pettis-Mackey}
Let $E$ be an lcs with the $(sDP)$ property such that $E_w$ is  sequentially  angelic.
\begin{enumerate}
\item[{\rm (i)}] Every $\sigma(E',E'')$-null sequence $\{ \chi_{n}\}_{n\in\NN}$ in $E'$ converges to zero also in $\mu(E',E)$. Consequently, every $\sigma(E',E'')$-sequentially compact subset of $E'$ is also $\mu(E',E)$-sequentially compact.
\item[{\rm (ii)}]  If additionally   $E$ has the Grothendieck property, then   $\big( E', \mu(E',E)\big)$ and $\big( E',\sigma(E',E'')\big)$ have the same sequentially compact sets. Moreover, the space $H:=\big( E', \mu(E',E)\big)$ has the Schur property.
\end{enumerate}
\end{proposition}
\begin{proof}
(i) (cf. \cite[19.18(c)]{Dom}) Suppose for a contradiction that $\chi_n \not\to 0$ in $\mu(E',E)$. Then there is a standard $\mu(E',E)$-neighborhood $[K;\e]=\{ \chi\in E': |\chi(x)|<\e\, \forall x\in K\}$ of zero in $E'$, where $K$ is an absolutely convex compact subset of $E_w$, which does not contain a subsequence $\{ \chi_{n_k}\}_{k\in\NN}$ of $\{ \chi_{n}\}_{n\in\NN}$. Therefore, for every $k\in\NN$ one can find $x_k\in K$ such that $\big| \chi_{n_k}(x_k)\big| \geq \e$. Since $E_w$ is  sequentially  angelic, $K$ is sequentially compact and hence we can assume that $x_k$ weakly converges to $x\in K$. Then the $(sDP)$ property implies
\[
\e \leq \big| \chi_{n_k}(x_k)\big|\leq \big| \chi_{n_k}(x)\big| + \big| \chi_{n_k}(x_k-x)\big| \to 0,
\]
a contradiction.

(ii) By (i) we have to show only that every sequentially compact subset $K$ of $\big( E', \mu(E',E)\big)$ is also $\sigma(E',E'')$-sequentially compact. Let $S=\{\chi_{n}: n\in\NN\}$ be a sequence in $K$. Then there is a subsequence $S'=\{\chi_{n_k}: k\in\NN\}$ of $S$ which $\mu(E',E)$-converges to some $\chi\in K$. Then $\chi_{n_k}\to \chi$ in $\sigma(E',E)$  and hence also in $\sigma(E',E'')$ by the Grothendieck property. Thus $K$ is $\sigma(E',E'')$-sequentially compact.

To prove that $H$ has the Schur property let $S$ be a $\sigma(E',E)$-null sequence. Then $S$ is a $\sigma(E',E'')$-null sequence by the Grothendieck property. Therefore, by (i), $S$ converges to zero in   $\mu(E',E)$. Thus $H$ has the Schur property.  \qed
\end{proof}

Let $E$ be a locally convex space. Recall that a subset $A$ of $E'$ is called {\em $E$-limited} if
\[
\sup\big\{ |\chi(x_n)|: \chi\in A\big\} \to 0
\]
whenever $\{ x_n\}_{n\in\NN}$ is a weakly null sequence in $E$.
It is well known (see \cite[Exercise 3.12]{HMVZ}) that a Banach space $E$ has the Schur property if and only if the closed unit ball of the dual space $E'$ is an $E$-limited set. Below we generalize this result to barrelled spaces.
\begin{proposition} \label{p:Schur-E-bounded}
A barrelled space $(E,\tau)$ has the Schur property if and only if every $\sigma(E',E)$-bounded subset of $E'$ is an $E$-limited set.
\end{proposition}

\begin{proof}
Assume that $E$ has the Schur property. Let  $A$ be a $\sigma(E',E)$-bounded subset of $E'$ and let $\{ x_n\}_{n\in\NN}$ be a $\sigma(E,E')$-null sequence in $E$. Since $E$ is barrelled there is a $\tau$-neighborhood $U$ of zero in $E$ such that $A\subseteq U^\circ$. For every $\e>0$, by the Schur property, there is an $N\in\NN$ such that $x_n\in \e U$ for every $n>N$. Then
$
|\chi(x_n)|=\e\big| \chi\big((1/\e) x_n\big)\big| \leq \e \mbox{ for every } \chi\in A \mbox{ and } n>N.
$
Thus $A$ is $E$-limited.

Conversely, assume that every $\sigma(E',E)$-bounded subset of $E'$ is an $E$-limited set. Let $x_n\to 0$ in $\sigma(E,E')$ and let $U$ be an absolutely convex closed $\tau$-neighborhood of zero in $E$. Then there is an $N\in\NN$ such that
$
|\chi(x_n)|\leq 1 \mbox{ for every } \chi\in U^\circ \mbox{ and } n>N.
$
So $x_n\in U^{\circ\circ}=U$ for every $n>N$. Thus $x_n\to 0$ in $\tau$ and $E$ has the Schur property.  \qed
\end{proof}

Recall that a Tychonoff space $X$ is called an {\em $F$-space} if every cozero-set $A$ in $X$ is $C^\ast$-embedded. For  numerous equivalent conditions for a Tychonoff space $X$ being an $F$-space see \cite[14.25]{GiJ}. In particular, the Stone--\v{C}ech compactification $\beta \Gamma$ of a discrete space $\Gamma$ is a compact $F$-space.

In  general the Schur property does not imply weak respecting compactness. Using the non-reflexivity of Banach spaces $C(K)$, it is shown in \cite[Example~6 (p.~267)]{Wil} or \cite[Example~19.19]{Dom} that the space $H=\big( E', \mu(E',E)\big)$, where $E= \ell_\infty(\NN)=C(\beta \NN)$,  has the Schur property but does not weakly respect compactness. Below we generalize this result with a different proof.

\begin{proposition} \label{p:Schur-non-Glicksberg}
Let $K$ be an infinite compact $F$-space, $E=C(K)$ and set $H:=\big( E', \mu(E',E)\big)$. Then:
\begin{enumerate}
\item[{\rm (i)}] $H$ is a complete space and has the Schur property;
\item[{\rm (ii)}] $H$ does not weakly respect compactness.
\end{enumerate}
\end{proposition}

\begin{proof}
(i) The space $H$ is complete by Exercise 3.41 of \cite{fabian-10}. The Banach space $E$ has the $(sDP)$ property by Theorem 13.43 of \cite{fabian-10} and has the Grothendieck property by Corollary 4.5.9 of \cite{Dales-Lau}.  Proposition 3.108 of \cite{fabian-10} implies that $E_w$ is angelic.  Thus $H$ has the Schur property by (ii) of Proposition \ref{p:Dunford-Pettis-Mackey}.

(ii) Consider the closed unit ball $B^\ast$ in $E'$. The Alaoglu theorem implies that $B^\ast$ is a weakly compact subset of $H$. Now assuming that $B^\ast$ is a compact subset of $H$ we apply Grothendieck's theorem \cite[Theorem~3.11]{HMVZ} to get that $B^\ast$ is $E$-limited. Therefore $E$ has the Schur property by Proposition \ref{p:Schur-E-bounded}. But this contradicts Proposition \ref{p:Schur-Ck}. Thus $H$ does not weakly respect compactness.  \qed
\end{proof}

M.~Valdivia \cite{Valdivia-91} and P.~Domanski and L.~Drewnowski \cite{DomDrew} proved independently  that a  Fr\'{e}chet space $E$ does not contain $\ell_1$ if and only if every $\mu(E',E)$-null sequence in $E'$ is strongly convergent to zero. In \cite[Theorem~2.1]{ruess} W.~Ruess generalized this result. 
\begin{proposition}[Ruess] \label{p:Dom-Drew}
Let $E$ be a locally complete  lcs whose every separable bounded set is metrizable. Then $E$ does not contain an isomorphic copy of $\ell_1$ if and only if every $\mu(E',E)$-null sequence in $E'$ is strongly convergent to zero. 
\end{proposition}

It is well known (see \cite[Proposition~11.6.2]{Jar}) that an lcs $E$ is Montel if and only if it is a quasi-complete quasibarrelled (equivalently, barrelled see \cite[Corollary~5.1.10]{bonet}) space and every equicontinuous set in $E'$ is relatively compact for $\beta(E',E)$.
It is proved in \cite{BLV} (see also \cite[Theorem~9]{Bonet-Lin-93}) that a Fr\'{e}chet space $E$ is Montel if and only if every $\sigma(E',E)$-convergent sequence in $E'$ is $\beta(E',E)$-convergent. Using Proposition \ref{p:Dom-Drew} we obtain a similar result for a wider class of locally convex spaces.
\begin{proposition} \label{p:Rosen-Montel}
Let $E$ be an lcs whose every separable bounded set is metrizable. Then $E$ is a Montel space  if and only if the following three conditions hold:
\begin{enumerate}
\item[{\rm (i)}] $E$ is a quasi-complete quasibarrelled space;
\item[{\rm (ii)}]  every $\sigma(E',E)$-convergent sequence in $E'$ is $\beta(E',E)$-convergent;
\item[{\rm (iii)}]  $E$ has the Schur property.
\end{enumerate}
\end{proposition}

\begin{proof}
Assume that $E$ is a Montel space. Then $E$  is a quasi-complete barrelled space by Proposition 11.6.2 of \cite{Jar}, and $E$ and $E'_\beta$ have the Schur property by Proposition \ref{p:Schur-Montel-respected}. The reflexivity of $E$ and the Schur property of $E'_\beta$ imply (ii).

Conversely, assume that (i)-(iii) hold. Since $E$ is quasibarrelled it is sufficient to show that $E$ is semi-Montel. To this end, we have to prove that every closed bounded subset $A$ of $E$ is compact. Since $E$ is quasi-complete, it suffices to show that $A$ is precompact, see Theorem 3.4.1 of \cite{NaB}. Suppose for a contradiction that $A$ is not precompact. Then, by Theorem 5 of \cite{BGP}, there is a neighborhood $U$ of zero and a sequence $S=\{ a_n: n\in\NN\}$ in $A$ such that
\begin{equation} \label{equ:Rosen-Montel-1}
a_n - a_m \not\in U \; \mbox{ for every } n\not= m.
\end{equation}
Proposition \ref{p:Dom-Drew} and (ii) imply that $E$ does not contain $\ell_1$. Therefore, by the Rosenthal property (Proposition \ref{p:Schur-Ruess}), $S$ contains a weakly Cauchy subsequence $\{ a_{n_k}: k\in\NN\}$. In particular $a_{n_{k+1}} -a_{n_k}\to 0$ in the weak topology, and hence $a_{n_{k+1}} -a_{n_k}\to 0$ in the topology of $E$ by the Schur property. But this contradicts (\ref{equ:Rosen-Montel-1}). Thus $E$ is semi-Montel.  \qed
\end{proof}

The next proposition extends (a) of Proposition 11 in \cite{Bonet-Lin-93} and has a similar proof.
\begin{proposition} \label{p:dual-Schur-Schur}
Let $E$ be a $c_0$-barrelled quasi-complete Mackey space whose every bounded set is metrizable. If $E$ has the Schur property, then also $\big( E', \mu(E',E)\big)$ has the Schur property. The converse holds if $U^\circ$ is $\sigma(E',E)$-sequentially compact for every neighborhood $U$ of zero in $E$, in particular, if $E$ is separable.
\end{proposition}

\begin{proof}
It follows from Theorem \ref{t:respecting-quasi-complete-metriz} and Proposition \ref{p:Schur-Ruess} that $E$ weakly respects compactness. Now Proposition \ref{p:Krein-characterization} implies that $(E',\mu(E',E))$ is a subspace of $\CC(E,\mathbf{F})$, so $\mu(E',E)=\tau_k|_{E'}$.

Let $\{\chi_n:n\in\NN\}$ be a $\sigma(E',E)$-null sequence in $E'$. Then the $c_0$-barrelledness of $E$ implies that the sequence $S=\{\chi_n:n\in\NN\}\cup\{ 0\}$ is equicontinuous. So $S\subseteq U^\circ$ for some neighborhood $U$ of zero in $E$. Since $\sigma(E',E)|_{U^\circ}=\tau_k|_{U^\circ}$ by Proposition 9.3.8 of \cite{horvath}, we obtain $\sigma(E',E)|_{U^\circ}=\mu(E',E)|_{U^\circ}$. Thus $\chi_n\to 0$ in $\mu(E',E)$.

Now assume that $U^\circ$ is $\sigma(E',E)$-sequentially compact for every neighborhood $U$ of zero in $E$ and every $\sigma(E',E)$-convergent sequence in $E'$ is $\mu(E',E)$-convergent. Suppose for a contradiction that $E$ is not a Schur space. So there is a weakly null sequence $S=\{x_n:n\in\NN\}$ in $E$ such that $x_n\not\to 0$ in $E$. Hence there exists an absolutely convex closed neighborhood $U$ of zero in $E$ such that $S\setminus U$ is infinite. By passing to a subsequence we can assume that $S\cap U=\emptyset$. For every $n\in\NN$ choose $\chi_n\in U^\circ$ such that $|\chi_n(x_n)|>1$. Since $U^\circ$ is $\sigma(E',E)$-sequentially compact, by passing to a subsequence if needed, we can additionally assume that $\chi_n$ weakly converges to some $\chi\in U^\circ$. By the Schur property we obtain $\chi_n \to \chi$ in $\mu(E',E)$.
Since the set $K=\overline{\mathrm{acx}}(S)$ is weakly compact by Theorem 5.1.11 of \cite{bonet}, it follows that
\[
A_n:= \sup\big\{ |(\chi_n -\chi)(x)| : x\in K\big\} \to 0 \quad \mbox{ at } n\to \infty.
\]
Therefore
\[
1< \big| \chi_n(x_n)\big| \leq \big| (\chi_n -\chi)(x_n)\big| + \big| \chi(x_n)\big| \leq A_n +\big| \chi(x_n)\big| \to 0,
\]
a contradiction.  \qed
\end{proof}

\medskip
\noindent {\bf Acknowledgments}. I am deeply indebted to  Professor J. Bonet for useful remarks and sending me \cite{Bonet-Lin-93}. I wish to thank the referee for careful reading and useful suggestions.

\bibliographystyle{amsplain}

\begin{thebibliography}{10}

\bibitem{ABR}
A.A. Albanese, J. Bonet, W.J. Ricker, Grothendieck spaces with the Dunford--Pettis property, Positivity, \textbf{14} (2010), 145--164.

\bibitem{BG}
T. Banakh, S. Gabriyelyan,  On the $\CC$-stable closure of the class of (separable) metrizable spaces, Monatshefte Math.   \textbf{180} (2016), 39--64.

\bibitem{BGP}
T. Banakh, S. Gabriyelyan,  I. Protasov, On uniformly discrete subsets in uniform spaces and topological groups, Mat. Studii \textbf{45} (2016), 76--97.

\bibitem{bierstedt1}
K.D. Bierstedt, J. Bonet, Density conditions in Fr\'echet and (DF)-spaces, Rev. Mat. Complut. \textbf{2} (1989), 59--75.

\bibitem{BLV}
J. Bonet, M. Lindstr\"{o}m, M. Valdivia, Two theorems of Josefson--Nissenzweig type for Fr\'{e}chet spaces, Proc. Amer. Math. Soc. \textbf{117} (1993), 363--364.

\bibitem{Bonet-Lin-93}
J. Bonet, M. Lindstr\"{o}m, Convergent sequences in duals of Fr\'{e}chet spaces, pp. 391--404 in ``Functional Analysis'', Proc. of the Essen Conference, Marcel Dekker, New York, 1993.

\bibitem{BFV}
J. Borwein, M. Fabian, J. Vanderwerff, Characterizations of Banach spaces via convex and other locally Lipschitz functions, Acta Math. Vietnam. \textbf{22} (1997), 53--69.

\bibitem{Bourgain-Ros}
J. Bourgain, H.P. Rosenthal, Martingales valued in certain subspaces of $L^1$, Israel J. Math. \textbf{37} (1980), 54--75.

\bibitem{CGL-08}
H. Carri\'{o}n, P. Galondo, M.L. Louren\c{c}o, A stronger Dunford--Pettis property, Studia Math. \textbf{184} (2008), 205--216.


\bibitem{Dales-Lau}
H.G. Dales, F.K. Dashiell, Jr., A.T.-M. Lau, D. Strauss, \emph{Banach Spaces of Continuous Functions as Dual Spaces}, Springer, 2016.

\bibitem{Diaz-89}
J.C. D\'{\i}az, Montel subspaces in the countable projective limits of $L^p(\mu)$-spaces, Canad. Math. Bull. \textbf{32} (1989), 169--176.

\bibitem{Diestel}
J. Diestel, \emph{Sequences and Series in Banach Spaces}, Garduate text in Mathematics \textbf{92}, Springer, 1984.

\bibitem{DomDrew}
P. Domanski, L. Drewnowski, Fr\'{e}chet spaces of continuous vector-valued functions: Complementability in dual Fr\'{e}chet spaces and injectivity, Studia Math. \textbf{102} (1992), 257--267.


\bibitem{Dom}
X.~Dom\'{\i}nguez, \emph{Grupos Abelianos topol\'{o}gicos y sumabilidad}, Doctoral Dissertation, Universidad Complutense de Madrid, 2001.


\bibitem{DowlingFree-1}
P.N. Dowling, D. Freeman, C.J. Lennard, E. Odell, B. Randrianantoanina, B. Turett, A weak Grothendieck compactness principle, J. Funct. Anal. \textbf{263} (2012), 1378--1381.

\bibitem{Edwards}
R.E. Edwards, \emph{Functional Analysis}, Reinhart and Winston, New York, 1965.

\bibitem{Eng}
R.~Engelking,   \emph{General topology}, Panstwowe Wydawnictwo Naukowe, 1985.


\bibitem{fabian-10}
M. Fabian, P. Habala, P. H\'{a}jek, V. Montesinos, J. Pelant, V. Zizler, \emph{Banach space theory. The basis for linear and nonlinear analysis}, Springer, New York, 2010.

\bibitem{Gab-Top-Nul}
S.~Gabriyelyan, Topological properties of the group of the null sequences valued in an Abelian topological group, Topology Appl.  \textbf{207} (2016), 136--155.

\bibitem{Gabr-LCS-Ascoli}
S.~Gabriyelyan, On the Ascoli property for locally convex spaces, Topology Appl. \textbf{230} (2017), 517--530.

\bibitem{Gab-Respected}
S. Gabriyelyan, Maximally almost periodic groups and respecting properties, available in arXiv:1712.05521.


\bibitem{GKKLP}
S. Gabriyelyan, J.~K{\c{a}}kol, A. Kubzdela, M. Lopez Pellicer,  On topological properties of Fr\'{e}chet locally convex spaces with the weak topology, Topology Appl.  \textbf{192} (2015), 123--137.


\bibitem{GalHer}
J.~Galindo, S.~Hern\'{a}ndez, The concept of boundedness and the Bohr compactification of a $MAP$ abelian group, Fundamenta Math.  \textbf{159} (1999), 195--218.

\bibitem{GiJ}
L.~Gillman and M.~Jerison, \emph{Rings of continuous functions}, Van Nostrand, New York, 1960.


\bibitem{HMVZ}
P. H\'{a}jek, V. Montesinos, J. Vanderwerff, V. Zizler, \emph{Biorthogonal systems in Banach spaces}, Springer, 2008.


\bibitem{HGM}
S.~Hern\'{a}ndez, J.~Galindo, S.~Macario, A characterization of the Schur property by means of the Bohr topology, Topology Appl. \textbf{97} (1999), 99--108.

\bibitem{HM}
S.~Hern\'{a}ndez, S.~Macario, Invariance of compactness for the Bohr topology, Topology Appl. \textbf{111} (2001), 161--173.

\bibitem{horvath}
J. Horv\'{a}th, \emph{Topological Vector Spaces and Distributions, I}. Addison-Wesley, Reading, Mass, 1966.

\bibitem{JPZ}
J.A. Jaramillo, A. Prieto, I. Zalduendo, Sequential convergences and Dunford--Pettis properties, Ann. Acad. Sci. Fenn. Math. \textbf{25} (2000), 467--475.

\bibitem{Jar}
H.~Jarchow, \emph{Locally Convex Spaces}, B.G. Teubner, Stuttgart, 1981.

\bibitem{JLO}
W.B. Johnson, R. Lillemets, E. Oja, Representing completely continuous operators through weakly $\infty$-compact operators, Bull. London Math. Soc. \textbf{48} (2016), 452--456.

\bibitem{Kothe}
G.~K\"{o}the, \emph{Topological vector spaces}, Vol. I, Springer-Verlag, Berlin, 1969.

\bibitem{LT-1}
J. Lindenstrauss, L. Tzafriri, \emph{Classical Banach Spaces I}, Springer, Berlin, 1977.


\bibitem{NaB}
L. Narici, E. Beckenstein,  \emph{Topological vector spaces}, Second Edition, CRC Press, New York, 2011.


\bibitem{bonet}
P. P\'{e}rez Carreras, J. Bonet, \emph{Barrelled Locally Convex Spaces}, North-Holland Mathematics Studies \textbf{131}, Amsterdam, 1987.


\bibitem{Pethe-Thakare}
P. Pethe, N. Thakare, Note on Dunford--Pettis Property and Schur Property, Indiana Univ. Math. J. \textbf{27} (1978), 91--92.


\bibitem{Pryce}
J.D.~Pryce, A device of R.J.~Whitley's applied to pointwise compactness in spaces of continuous functions, Proc. London Math. Soc. \textbf{23} (1971), 532--546.

\bibitem{Rosenthal}
H.P. Rosenthal, A characterization of Banach spaces containing $\ell^1$, Proc. Nat. Acad. Sci. U.S.A. \textbf{71} (1974), 2411--2413.


\bibitem{ruess}
W. Ruess, Locally convex spaces not containing $\ell_{1}$, Funct. Approx. Comment. Math. \textbf{50} (2014), 351--358.


\bibitem{Valdivia-77}
M. Valdivia, Some new results on weak compactness, J. Funct. Anal. \textbf{24} (1977), 1-10.

\bibitem{Valdivia-91}
M. Valdivia, On totally reflexive Fr\'{e}chet spaces, pp. 39--55 in Actti del Convegno Internazionale in Analisi Matematico e sue Applicazioni, Dedicato Prof. G. Aquaro. Bari 1991.

\bibitem{Wil}
A.~Wilansky, \emph{Modern methods in topological vector spaces}, McGraw-Hill, 1978.

\end{thebibliography}

\end{document}